\documentclass[11pt,reqno]{amsart}
\usepackage{amsmath, amssymb, amsthm}
\usepackage{url}
\usepackage[breaklinks]{hyperref}
\usepackage[czech, english]{babel}
\usepackage{eufrak}
\usepackage[top=1in, bottom=1in, left=1in, right=1in]{geometry}
\usepackage{xcolor}

\renewcommand{\(}{\left\(}
\renewcommand{\)}{\right\)}
\renewcommand{\[}{\left\[}
\renewcommand{\]}{\right\]}
\renewcommand{\i}{\infty}

\numberwithin{equation}{section}
 \theoremstyle{plain}
\newtheorem{theorem}{Theorem}[section]
\newtheorem{lemma}[theorem]{Lemma}

\newtheorem*{remark*}{Remark}
\newtheorem{conjecture}[theorem]{Conjecture}

\newtheorem{corollary}[theorem]{Corollary}

   \makeatletter
\def\proof{\@ifnextchar[{\@oproof}{\@nproof}}
\def\@oproof[#1][#2]{\trivlist\item[\hskip\labelsep\textit{#2 Proof of\
#1.}~]\ignorespaces}
\def\@nproof{\trivlist\item[\hskip\labelsep\textit{Proof.}~]\ignorespaces}

\makeatother

\begin{document}
\title[A modular relation and the GRH]{A modular relation involving non-trivial zeros of the Dedekind zeta function, and the Generalized Riemann Hypothesis} 

\author{Atul Dixit, Shivajee Gupta, Akshaa Vatwani}\thanks{2020 \textit{Mathematics Subject Classification.} Primary 11M41, Secondary 11R42, 33C10.\\
	\textit{Keywords and phrases.} Dedekind zeta function, Generalized Riemann Hypothesis, Riesz-type criterion, modular relation, Bessel function.}
\address{Discipline of Mathematics, Indian Institute of Technology Gandhinagar, Palaj, Gandhinagar 382355, Gujarat, India} 
\email{adixit@iitgn.ac.in, shivajee.o@iitgn.ac.in, akshaa.vatwani@iitgn.ac.in \newline }
\begin{abstract}
We give a number field analogue of a result of Ramanujan, Hardy and Littlewood, thereby obtaining a modular relation involving the non-trivial zeros of the Dedekind zeta function. We also provide a Riesz-type criterion for the Generalized Riemann Hypothesis for  $\zeta_{\mathbb{K}}(s)$. New elegant transformations are obtained when $\mathbb{K}$ is a quadratic extension, one of which involves the modified Bessel function of the second kind. 
\end{abstract}
\maketitle
\vspace{-1cm}
\tableofcontents
\vspace{-1.3cm}
\section{Introduction}\label{intro}
Let $\mathbb{K}$ be an algebraic number field, $\mathcal{O}_\mathbb{K}$ be the ring of integers of $\mathbb{K}$.  The Dedekind zeta function of $\mathbb{K}$ is defined by 
\begin{equation*}
	\zeta _\mathbb{K}(s):=\sum_{\mathfrak{a} \ \subseteq \ \mathcal{O}_\mathbb{K} }\frac{1}{\mathcal{N}(\mathfrak{a})^s}\hspace{5mm}(\textup{Re}(s)>1),
\end{equation*}
where the sum runs over non-zero integral ideals $\mathfrak{a}$ of  $\mathbb{K}$ and $\mathcal{N}(\mathfrak{a})$ denotes the norm of $\mathfrak{a}$.
Equivalently,
\begin{equation*}
	\zeta _\mathbb{K}(s)= \sum_{n=1}^{\infty}\frac{a_n}{n^s},
\end{equation*}
where $a_n$ denotes the number of integral ideals with norm $\mathcal{N}(\mathfrak{a})=n$. It is known that $\zeta _\mathbb{K}(s)\neq 0$ for Re$(s)>1$, and \cite[p.~172]{densde}
\begin{equation}\label{defb}
	\frac{1}{\zeta _\mathbb{K}(s)}= \sum_{n=1}^{\infty}\frac{b_n}{n^s} \hspace{5mm}(\textup{Re}(s)>1),
\end{equation}
where ${\displaystyle  b_n =\sum\limits_{\mathop {\nu {}}\limits_{\mathfrak{a}\in S} } {(-1)^{\nu}}}$ and 
$
S = \{\mathfrak{a} :\mathcal{N}(\mathfrak{a})=n, \ \mathfrak{a}=\mathfrak{p}_1.\mathfrak{p}_2 \cdots \mathfrak{p}_\nu, \textup{ where } \mathfrak{p}_i \textup{ are distinct prime ideals}\}$. Moreover, we have \cite[p.~89]{Landau} 
\begin{align*}
	\sum_{n=1}^{\infty} \frac{b_n}{n}=0.
\end{align*}
The	Dedekind zeta function has an analytic continuation to the entire complex plane except for a simple pole at $s=1$. The completed Dedekind zeta function is defined by \cite[p.~257]{mar}
\begin{equation*}
	{\Lambda _\mathbb{K}}(s) := {\left( {\frac{{|{d_\mathbb{K}}|}}{{{4^{{r_2}}}{\pi ^n}}}} \right)^{\frac{s}{2}}}{\Gamma ^{{r_1}}}\left( {\frac{s}{2}} \right){\Gamma ^{{r_2}}}(s){\zeta _\mathbb{K}}(s),
\end{equation*}
where $d_\mathbb{K}$ is the discriminant of the field $\mathbb{K}$, $r_1$ and $2r_2$ are the number of real and complex embeddings of $\mathbb{K}$  respectively. 
The functional equation of $\zeta_{\mathbb{K}}(s)$ is given for all $s$ by \cite[p.~266]{mar}
\begin{align}\label{feq}
		{\Lambda _\mathbb{K}}(s) &={\Lambda _\mathbb{K}}(1-s). 
\end{align}

In this paper, we obtain a modular relation involving the non-trivial zeros of $\zeta_{\mathbb{K}}(s)$,  the special case $\mathbb{K}=\mathbb{Q}$ of which was given by Ramanujan, Hardy, and Littlewood. We begin with a historical account of the latter and then discuss some recent results in this area.

After going to England, Ramanujan informed Hardy about a striking modular transformation involving infinite series of the M\"{o}bius function \cite[p.~468, Entry 37]{berndt1}. Hardy and Littlewood \cite[p.~156, Section 2.5]{hl} later gave a corrected version of this transformation which, till date, has not been proved rigorously, hence known as the \emph{Ramanujan-Hardy-Littlewood conjecture}. This is stated next.

	\begin{conjecture}[Ramanujan-Hardy-Littlewood]
Let $\alpha$ and $\beta$ be two positive numbers such that $\alpha\beta=\pi$. Assume that the series $\sum_{\rho}\left(\Gamma{\left(\frac{1-\rho}{2}\right)}/\zeta^{'}(\rho)\right)x^{\rho}$ converges for every positive real $x$, where $\rho$ runs through the non-trivial zeros of $\zeta(s)$, and that the non-trivial zeros of $\zeta(s)$ are simple. Then
\begin{align}\label{mrhl}
&\sqrt{\alpha}\sum_{n=1}^{\infty}\frac{\mu(n)}{n}e^{-\alpha^2/n^2}-\sqrt{\beta}\sum_{n=1}^{\infty}\frac{\mu(n)}{n}e^{-\beta^2/n^2}=-\frac{1}{2\sqrt{\beta}}\sum_{\rho}\frac{\Gamma{\left(\frac{1-\rho}{2}\right)}}{\zeta^{'}(\rho)}\beta^{\rho}.
\end{align}
	\end{conjecture}
This result is still a conjecture because the convergence of the series on the right-hand side above is not known, even upon assuming the Riemann Hypothesis (RH)! At this point of time, we know that the series converges only if we bracket the terms of the series in such a way that the terms for which 
\begin{equation*}
|\text{Im}\;\rho-\text{Im}\;\rho'|<\exp\left(-c\;\text{Im}\;\rho/\log (\text{Im}\;\rho)\right)+\exp\left(-c\;\text{Im}\;\rho'/\log (\text{Im}\;\rho')\right)
\end{equation*}
are included in the same bracket (see \cite[p.~220]{titch}). But as Hardy and Littlewood say in a footnote of their paper \cite[p.~159]{hl}, one does not know anything about the size of these brackets. It is, however, believed that the series is not merely convergent (that too without bracketing terms), but rapidly convergent.

Equation \eqref{mrhl} motivated Hardy and Littlewood \cite{hl} to obtain a Riesz-type criterion for RH:
\begin{theorem}
Consider the function $P(\beta):=\displaystyle\sum_{m=1}^{\infty}\frac{(-\beta)^m}{m!\zeta(2m+1)}$. Then, the estimate $P(\beta)=O_{\delta}\big(\beta^{-\frac{1}{4}+\delta}\big)$ as $\beta\to\infty$ for all positive values of $\delta$ is equivalent to the Riemann Hypothesis.
\end{theorem}
This criterion is called so because Riesz \cite{riesz0} was the first mathematician to obtain a result of this type.

There exist several analogues and generalizations of \eqref{mrhl} in different directions. For example, the analogue of \eqref{mrhl} in the setting of Dirichlet $L$-functions is given in \cite{charram}. A generalization of \eqref{mrhl} containing an extra complex variable $z$ is derived in \cite{dixthet}. An application of this generalization towards obtaining a generalized Riesz-type criterion for RH is given in \cite{riesz} along with its corresponding analogue for Dirichlet characters. The analogue of the aforementioned generalization of \eqref{mrhl} for Hecke forms is given in \cite{hecke} along with its application towards obtaining a Riesz-type criterion for the Riemann Hypothesis for $L$-functions attached to primitive Hecke forms. In \cite{rzz}, Roy, Zaharescu, and Zaki obtain a result of the type in \eqref{mrhl} where the M\"{o}bius function is replaced by a convolution of Dirichlet characters with the M\"{o}bius function. A plethora of results of the type in \eqref{mrhl} have been obtained by K\"{u}hn, Robles and Roy in \cite{krr} for functions reciprocal (and also self-reciprocal) in the Hankel kernel using their main theorem \cite[Theorem 1.2]{krr} which is valid for any zeta function of degree $1$ in the Selberg class. Recently, Agarwal, Garg and Maji \cite{agm} have obtained a one-variable generalization of \eqref{mrhl} in a different direction.

As remarked before, in this paper, we obtain a generalization of \eqref{mrhl} in the setting of the Dedekind zeta function. We also obtain a corresponding Riesz-type criterion for the Generalized Riemann Hypothesis for $\zeta_{\mathbb{K}}(s)$ (GRH).

A novel feature of our work is that in both our results, we come across an extra expression which was not present in any of the analogues or generalizations mentioned in the preceding paragraph.  In our modular relation involving
the Dedekind zeta function, these expressions appear due to a pole of a
combination of Gamma factors and $\zeta _{\mathbb{K}}(s)$ at $s=1$ of order $r:=r_1+r_2-1$, and due to the zero of $\zeta _{\mathbb{K}}(s)$ at $s=0$ of order $r$. The
poles of the associated integrand at $s=0$ and $s=1$ now play a non-trivial role owing
to the fact that for $\mathbb K$ of degree $n\geq 2$, we have
$r\geq 1$. In our Riesz-type criterion for the GRH, an additional expression similarly comes up because of a pole of order $r_1+r_2$ at $s=0$, arising from a combination of Gamma factors. These two results are stated in Theorems \ref{rhldede} and \ref{rtc} respectively. 
 It is worth noting that  in the $r=0$ case, the additional term in Theorem \ref{rhldede} does not appear, whereas for  Theorem  \ref{rtc}, the corresponding term can be shown to be of the same order of magnitude as the error term. In both cases, we are therefore able to recover the previously known results on $\zeta(s)$, namely, \eqref{mrhl} and the Riesz-type criterion for the RH.  
 


\begin{theorem}\label{rhldede}
Let $\mathbb{K}$ be an algebraic number field with discriminant $d_\mathbb{K}$. Let $\left[\mathbb{K}:\mathbb{Q} \right] =n$. Assume the convergence of the series $\sum\limits_\rho  {{x ^{ \rho }}\frac{{{\Gamma ^{{r_1}}}\left( {\frac{{1 - \rho }}{2}} \right){\Gamma ^{{r_2}}}\left( {1 - \rho } \right)}}{{{\zeta _\mathbb{K}}^\prime (\rho )}}} $ for every positive real $x$, where $\rho$ runs through the non-trivial zeros of $\zeta_\mathbb{K} (s)$. Suppose that the multiplicity of each non-trivial zero of $\zeta_\mathbb{K}(s)$ is $1$. Let $\alpha$, $\beta$ be positive numbers such that $\alpha \beta =\eta$, where $\eta=\frac{4^{r_2}\pi^n}{|d_\mathbb{K}|}$. Then     
{\allowdisplaybreaks	\begin{align}
		\label{rhldedeeqn}
		& \sqrt{\alpha}\sum \limits _{n = 1}^\infty {\frac{{{b_{n}}}}{n}{Z_{{r_{1}},{r_{2}}}}
			\left ( {\frac{\alpha }{n}} \right )} - \sqrt{\beta}\sum \limits _{n =
			1}^\infty {\frac{{{b_{n}}}}{n}{Z_{{r_{1}},{r_{2}}}}\left ( {
				\frac{\beta }{n}} \right )}
		\nonumber
		\\
		&\qquad =- {\left . {\frac{1}{{\sqrt{\beta}(r - 1)!}}
				\frac{{{d^{r - 1}}}}{{d{s^{r - 1}}}}{{\left ( {s - 1} \right )}^{r}}{
					\beta ^{ s}}
				\frac{{{\Gamma ^{{r_{1}}}}\left ( {\frac{{1 - s}}{2}} \right ){\Gamma ^{{r_{2}}}}\left ( {1 - s} \right )}}{{{\zeta _\mathbb{K}}(s)}}}
			\right |_{s = 1}}
		\nonumber
		\\
		&\qquad\quad-{\left. {\frac{1}{{\sqrt{\beta}(r - 1)!}}\frac{{{d^{r - 1}}}}{{d{s^{r - 1}}}}{{ {s} }^r}{\beta ^{ s}}\frac{{{\Gamma ^{{r_1}}}\left( {\frac{{1 - s}}{2}} \right){\Gamma ^{{r_2}}}\left( {1 - s} \right)}}{{{\zeta _\mathbb{K}}(s)}}} \right|_{s = 0}} \nonumber
		\\
		& \quad \qquad - {\frac{1}{\sqrt{\beta}}}\sum \limits _\rho {{\beta ^{
					\rho }}
			\frac{{{\Gamma ^{{r_{1}}}}\left ( {\frac{{1 - \rho }}{2}} \right ){\Gamma ^{{r_{2}}}}\left ( {1 - \rho } \right )}}{{{\zeta _\mathbb{K}}^\prime (\rho )}}},
	\end{align}}
where, for $-{1\over 2}<\textup{Re}(s)=c<0$, \begin{equation}\label{defz}
	{Z_{{r_1},{r_2}}}(x) =\frac{1}{2 \pi i} \int_{c-i\infty}^{c+i\infty} {{\Gamma ^{{r_1}}}\left( {\frac{s}{2}} \right){\Gamma ^{{r_2}}}\left( s \right){x^{ - s}}ds}. 
\end{equation} 
\end{theorem}
\begin{theorem}\label{rtc}
	Let $\mathbb{K}$ be a algebraic number field and $b_n$ be as defined in \eqref{defb}. Let $\mathcal{P} _{r_1,r_2}(y)$ be defined by \begin{align}\label{defp}
		\mathcal{P} _{r_1,r_2}(y):=\sum\limits_{n = 1}^\infty  {\frac{{{b_n}}}{n}{Z_{{r_1},{r_2}}}\left( {\frac{\sqrt{y} }{n}} \right)} \hspace{15pt}(y>0).
	\end{align} 
With $r=r_1+r_2-1$, we have the following:
	\begin{itemize}
		\item[(1)] The estimate $\mathcal{P} _{r_1,r_2}(y)=O_{r_1,r_2}\left( y^{-\frac{1}{4}+\delta} \right)$ as $y \rightarrow \infty$ for all $\delta >0$ implies the Generalized Riemann Hypothesis for $\zeta_\mathbb{K}(s)$.
		\item[(2)] \begin{itemize}
			\item[(a)] If $r\neq 0$ and $\epsilon >0 $, the Generalized Riemann Hypothesis for $\zeta_\mathbb{K}(s)$ implies  
			\begin{align}\label{estp}
				\mathcal{P} _{r_1,r_2}(y)=-\frac{2^{r_2}}{r!}\sum_{n=1}^{[y ^{{1\over 2}-\epsilon}] -1}{b_n \over n}\sum_{i=0}^{r}C_i\left(\begin{matrix}
					r\\ i
				\end{matrix} \right)\left(\log \left({n \over \sqrt{y}} \right)  \right)^{r-i}  +O_{r_1,r_2}\left( y^{-\frac{1}{4}+\delta} \right)
			\end{align} as $y \rightarrow \infty$ for all $\delta >0$. Here $C_i=X^{(i)}_{r_1,r_2}(0)$, and $X^{(i)}_{r_1,r_2}(s)$ denotes the $i^{th}$ derivative of  $\Gamma^{r_1}\left( {s\over 2}+1\right) \Gamma^{r_2}(s+1)$.
			\item[(b)] If $r=0$, then the Generalized Riemann Hypothesis for $\zeta_\mathbb{K}(s)$ implies $\mathcal{P} _{r_1,r_2}(y)=O_{r_1,r_2}\left( y^{-\frac{1}{4}+\delta} \right)$ as $y \rightarrow \infty$ for all $\delta >0$. 
		\end{itemize}
	\end{itemize}
\end{theorem}

\begin{remark*}
	When $\mathbb{K}$ is $\mathbb{Q}$ or an imaginary quadratic field, we have $r=0$. In particular, the estimate for $\mathcal{P} _{r_1,r_2}(y)$ in \textup{(1)} is then equivalent to the GRH.
\end{remark*}

\section{Preliminary results}\label{prelim}
For ${1\over 2}\leq \delta \leq 1$, let ${N_\mathbb{K}}({\delta,T}) $ denote the number of zeros $\rho =\sigma +i\gamma$ of $\zeta _\mathbb{K} (s)$ with $|\gamma| \leq T,\ \delta \leq \sigma$. \\ Then from \cite[Equation (6)]{densde},
\begin{equation}\label{dens}
	{N_\mathbb{K}}\left( {\frac{1}{2},T + 1} \right) - {N_\mathbb{K}}\left( {\frac{1}{2},T} \right) \ll \log T.
\end{equation}
Using the functional equation \eqref{feq}, one can conclude that the  number of zeros of $\zeta _\mathbb{K}(s)$ in the critical strip between the horizontal lines Im$(s)=T+1$ and Im$(s)=T-1$ is also $O(\log T)$.

We define a function ${\tilde{Z}_{{r_1},{r_2}}}:\mathbb{R} \rightarrow \mathbb{C}$ by 
\begin{equation}\label{defzz}
	{\tilde{Z}_{{r_1},{r_2}}}(x) :=\frac{1}{2 \pi i} \int\limits_{\left( {d} \right)} {{\Gamma ^{{r_1}}}\left( {\frac{s}{2}} \right){\Gamma ^{{r_2}}}\left( s \right){x^{ - s}}ds}  \hspace{15pt} (d>0).
\end{equation}
Let $r=r_1+r_2-1$ and $n=[\mathbb{K}:\mathbb{Q}]$. We will need the asymptotic estimate 
 \begin{equation}\label{estz}
 	\tilde{Z}_{{r_1},{r_2}}(x) \ll_{r_1, r_2} x^{-{r\over n}}  \exp \bigg( -n \left( {x\over 2^{r_2} }\right) ^{2 \over n} \bigg), 
 \end{equation}
 as $x\rightarrow \infty$, in the proof of Theorem \ref{rtc}. This can be derived using the asymptotic estimate of the Meijer G-function given in \cite[p. 180]{luke}.
 
 Stirling's formula for $\Gamma(s)$, $s=\sigma+it$, in a vertical strip $C\leq\sigma\leq D$ is given by \cite[p.~224]{cop}
 \begin{equation}\label{strivert}
 	|\Gamma(s)|=(2\pi)^{\tfrac{1}{2}}|t|^{\sigma-\tfrac{1}{2}}e^{-\tfrac{1}{2}\pi |t|}\left(1+O\left(\frac{1}{|t|}\right)\right),
 \end{equation}
 as $|t|\to\infty$.
 
Corollary \ref{cor2} of Theorem \ref{rhldede} involves $K_\nu(z)$, the modified Bessel function of the second kind, defined below.

 The Bessel function of the first kind of order $\nu$ is defined by \cite[p.~40]{watson}
 	\begin{align*}
 		J_{\nu}(z)&:=\sum_{m=0}^{\infty}\frac{(-1)^m(z/2)^{2m+\nu}}{m!\Gamma(m+1+\nu)} \hspace{9mm} (z,\nu\in\mathbb{C}),
 	\end{align*}
 	The modified Bessel functions of the first and second kinds of order $\nu$ are defined by \cite[pp.~77-78]{watson}
 	\begin{align*}
 		I_{\nu}(z)&:=
 		\begin{cases}
 			e^{-\frac{1}{2}\pi\nu i}J_{\nu}(e^{\frac{1}{2}\pi i}z), & \text{if $-\pi<$ arg $z\leq\frac{\pi}{2}$,}\\
 			e^{\frac{3}{2}\pi\nu i}J_{\nu}(e^{-\frac{3}{2}\pi i}z), & \text{if $\frac{\pi}{2}<$ arg $z\leq \pi$,}
 		\end{cases}\\
 		K_{\nu}(z)&:=\frac{\pi}{2}\frac{I_{-\nu}(z)-I_{\nu}(z)}{\sin\nu\pi}
 	\end{align*}
 	respectively, with $K_n(z)$ defined by $\lim_{\nu\to n}K_{\nu}(z)$ if $n$ is an integer. 
 	The modified Bessel function of order zero can also be written as an inverse Mellin transform, that is, for $d=\textup{Re}(s)>0$,
 	\begin{align*}
 		K_0(x)={1 \over 2\pi i}\int_{(d)}\Gamma^2\left({s \over 2}\right)2^{s-2}x^{-s}ds. 
 	\end{align*}
 	Moreover, by the residue theorem, for $-1<c=\textup{Re}(s)<0$,
 	\begin{align}\label{bessel}
 		{1 \over 2\pi i}\int_{(c)}\Gamma^2\left({s \over 2}\right)2^{s-2}x^{-s}ds=K_0(x)+\gamma+\log \left({x \over 2} \right) ,
 \end{align}
where $\gamma$ denotes Euler's constant.
 

\section{A modular relation involving the Dedekind zeta function}\label{mr}
Before giving the proof of Theorem \ref{rhldede}, we prove the following lemma.
\begin{lemma}\label{lemma}
  Let $T\rightarrow \infty$ through values such that $|T-\gamma|>\exp{ \left(-A_1 \gamma / \log\gamma \right)}$ for every ordinate $\gamma $ of a zero of $\zeta _\mathbb{K}(s)$, where $A_1$ is a sufficiently small positive constant. Then for $\sigma \in [{-1/ 2},{3/ 2} ]$, we have 
  \begin{equation}
	|\zeta _\mathbb{K}(\sigma + iT)| \geq e^{-A_2 T},	
\end{equation}
where $0<A_2 < {\pi\over4}$.
\end{lemma}	
\begin{proof}
	Let $s=\sigma +it$ and $\rho=\delta+i\gamma$, where $\sigma \in [{-1/ 2},{3/ 2} ]$ and $t$ be a fixed positive number greater than 2.
	From \cite[p. 277-279]{mar},
	\begin{equation}\label{esdede1}
		\frac{\zeta '_\mathbb{K} (s)}{\zeta _\mathbb{K} (s)}={\sum_\rho\left( \frac{1}{s-\rho}+\frac{1}{\rho}\right)}+O(\log t),
	\end{equation}
where the sum runs over the non-trivial zeros of $\zeta_\mathbb{K}(s)$.
	 Substituting $s$ by ${3\over 2}+it$ and then substracting the resulting equation from \eqref{esdede1}, we obtain
\begin{align}\label{le1}
	\frac{\zeta '_\mathbb{K} (s)}{\zeta _\mathbb{K} (s)}&={ \sum_\rho\left( \frac{1}{s-\rho}-\frac{1}{{3\over 2}+it-\rho}\right)}+O(\log t) \nonumber \\
	& ={ \sum_{|t-\gamma| \leq 1} \frac{1}{s-\rho}}-{ \sum_{|t-\gamma| \leq 1} \frac{1}{{3\over 2}+it-\rho}}+{ \sum_{|t-\gamma| \geq 1}\left( \frac{1}{s-\rho}-\frac{1}{{3\over 2}+it-\rho}\right)}+O(\log t).
\end{align}
From \eqref{dens}, we conclude that
\begin{align}\label{le2}
	{ \sum_{|t-\gamma| \leq 1} \frac{1}{{3\over 2}+it-\rho}}={ \sum_{|t-\gamma| \leq 1} O(1)}=O(\log t).
\end{align}
Also, for $n \in \mathbb{N}$,
\begin{align}\label{eq5}
	{\sum_{t+n<\gamma \leq t+n+1} \left( \frac{1}{s-\rho}-\frac{1}{{3\over 2}+it-\rho}\right)}&={\sum_{t+n<\gamma \leq t+n+1} \frac{{3\over 2}-\sigma}{(s-\rho)\left({3\over 2}+it-\rho \right)}}\nonumber \\  &\ll{ \sum_{t+n<\gamma \leq t+n+1} \frac{1}{(\gamma-t)^2}}\nonumber \\  &\ll{ \sum_{t+n<\gamma \leq t+n+1} \frac{1}{n^2}}\nonumber \\&\ll  \frac{\log (t+n)}{n^2}.
\end{align}
Since
 \begin{align*}
 	\sum_{n=1}^\infty \left( \frac{\log (t+n)}{n^2}\right) <\sum_{n\leq t}\frac{\log 2t}{n^2}+\sum_{n> t}\frac{\log 2n}{n^2}=O(\log t),
 \end{align*}
summing over $n$ from $1$ to $\infty$ on both sides of \eqref{eq5}, we see that
\begin{align*}
	{ \sum_{\gamma>t+1}\left( \frac{1}{(s-\rho)}-\frac{1}{{3\over 2}+it-\rho}\right)}=O(\log t),
\end{align*} 
and similarly that
\begin{align*}
	{ \sum_{\gamma < t-1}\left( \frac{1}{(s-\rho)}-\frac{1}{{3\over 2}+it-\rho}\right)}=O(\log t).
\end{align*}
Hence
 \begin{align}\label{iugfiuy}
	{ \sum_{|t-\gamma| \geq 1}\left( \frac{1}{(s-\rho)}-\frac{1}{{3\over 2}+it-\rho}\right)}=O(\log t).
\end{align}
Thus, from \eqref{le1}, \eqref{le2}, and \eqref{iugfiuy}, we have
\begin{align}\label{lemma1}
	\frac{\zeta '_\mathbb{K} (s)}{\zeta _\mathbb{K} (s)} ={ \sum_{|t-\gamma| \leq 1} \frac{1}{s-\rho}}+O(\log t).
\end{align}
Integrating \eqref{lemma1} with respect to $s$ from ${3\over 2} +it$ to $z$, where $z=\sigma' + it$ and $\sigma' \in [{-1/ 2},{3/ 2} ]$, we get 
\begin{align}\label{eq7}
	\log \zeta_\mathbb{K} (z) - \log \zeta_\mathbb{K} \left({3\over2} +it \right)&=\bigg[{\sum_{|t-\gamma| \leq 1} \log (s-\rho) } \bigg]_{{3\over2} +it} ^{{z}} +O(\log t) \nonumber \\
	&={\sum_{|t-\gamma| \leq 1} \log (z-\rho) }-{\sum_{|t-\gamma| \leq 1} \log \left( {3\over 2}+it-\rho\right)  } +O(\log t).
\end{align}
Since  $\log \left( {3\over 2}+it-\rho\right) $ is bounded when $|t-\gamma|\leq1$, from \eqref{dens}, we obtain \begin{equation*}
	{\displaystyle\sum_{|t-\gamma| \leq 1} \log \left({3\over 2}+it-\rho \right) }=O(\log t).
\end{equation*} Moreover, from  \cite[Equation (5)]{densde}, we have $\log \zeta_\mathbb{K}\left({3\over2}+it\right)=O(\log t)$. 
Consequently, \eqref{eq7} yields
\begin{align*}
	\log \zeta_\mathbb{K} (s) ={\sum_{|t-\gamma| \leq 1} \log (s-\rho) }+O(\log t).
\end{align*}
Since Re$(s-\rho)$ is bounded, we obtain
\begin{align*}
	\log |\zeta_\mathbb{K} (s)| \geq {\sum_{|t-\gamma| \leq 1} \log |t-\gamma| }+O(\log t).
\end{align*}
Letting $s$ be $\sigma + iT$ and using 
the bound on $|T-\gamma|$ assumed in the statement of the lemma, we have 
\begin{align}\label{lower1}
	\log |\zeta_\mathbb{K} (\sigma + iT)| \geq -{\sum_{|T-\gamma| \leq 1}{ A_1 \gamma \over \log\gamma }}+O(\log T),
\end{align}
where $A_1$ is a sufficiently small absolute constant to be chosen later. Now, from \eqref{dens}, 
\begin{align}\label{lower2}
	{\sum_{|T-\gamma| \leq 1} {A_1 \gamma \over \log\gamma }} \leq \sum_{|T-\gamma| \leq 1} A_1 \frac{T+1}{\log(T-1)} \leq CA_1T,
\end{align}
for some absolute constant $C>0$. From \eqref{lower1} and \eqref{lower2}, we have 
\begin{align*}
	\log |\zeta_\mathbb{K} (\sigma + iT)| \geq -CA_1T+O(\log T).
\end{align*}
Choosing $A_1$ sufficiently small so that $CA_1< {\pi \over 4}$ completes the proof.
\end{proof}

Throughout the sequel, we use the notation $\int_{(c)}$ to denote the line integral $\int_{c-i\infty}^{c+i\infty}$.

\subsection{Proof of Theorem \ref{rhldede}}
Using the definition \eqref{defz} of ${Z_{{r_1},{r_2}}}(x)$, we see that since $-{1\over 2}<c<0$,
{\allowdisplaybreaks\begin{align}\label{xr2}
	 \sum\limits_{n = 1}^\infty  {\frac{{{b_n}}}{n}{Z_{{r_1},{r_2}}}\left( {\frac{\alpha }{n}} \right)}  &= \sum\limits_{n = 1}^\infty  {\frac{{{b_n}}}{n}\frac{1}{{2\pi i}}\int\limits_{\left( {c} \right)} {{\Gamma ^{{r_1}}}\left( {\frac{s}{2}} \right){\Gamma ^{{r_2}}}\left( s \right){{\left( {\frac{\alpha }{n}} \right)}^{ - s}}ds} }   \nonumber \\	
	  &  = \frac{1}{{2\pi i}}\int\limits_{\left( {c} \right)} {{\Gamma ^{{r_1}}}\left( {\frac{s}{2}} \right){\Gamma ^{{r_2}}}\left( s \right){\alpha ^{ - s}}\left( {\sum\limits_{n = 1}^\infty  {\frac{{{b_n}}}{{{n^{1 - s}}}}} } \right)ds}\nonumber \\   	
	&   = \frac{1}{{2\pi i}}\int\limits_{\left( {c} \right)} {{\alpha ^{ - s}}\frac{{{\Gamma ^{{r_1}}}\left( {\frac{s}{2}} \right){\Gamma ^{{r_2}}}\left( s \right)}}{{{\zeta _\mathbb{K}}(1 - s)}}ds}  	\nonumber \\
	&= \frac{1}{{2\pi i}}\int\limits_{\left( {c} \right)} {{{\left( {\frac{{|{d_\mathbb{K}}|}}{{{4^{{r_2}}}{\pi ^n}}}} \right)}^{\frac{1}{2} - s}}{\alpha ^{ - s}}\frac{{{\Gamma ^{{r_1}}}\left( {\frac{{1 - s}}{2}} \right){\Gamma ^{{r_2}}}\left( {1 - s} \right)}}{{{\zeta _\mathbb{K}}(s)}}ds}, 	
	\end{align}}
where in the second step, we interchange the order of integration and summation with the help of Stirling's formula \eqref{strivert}.
For the last equality above, we have used the functional equation \eqref{feq}.

 Consider a positively oriented contour with sides $\left[c -iT, d-iT \right],  \left[d-iT, d + iT \right], \left[d + iT, c+iT \right]$ and $ \left[c+iT, c-iT \right] $, where $T>0$ and $d\in (1,{3 / 2})$. It is now necessary to consider poles of ${\Gamma ^{{r_1}}}\left( {\frac{{1 - s}}{2}} \right)$, ${\Gamma ^{{r_2}}}\left( {1 - s} \right)$ and $\zeta_\mathbb{K} (s)$. The poles of ${\Gamma ^{{r_1}}}\left( {\frac{{1 - s}}{2}} \right)$ occur at the odd natural numbers, each of order $r_1$. Poles of ${\Gamma ^{{r_2}}}\left( {1 - s} \right)$ occur at the natural numbers, each of order $r_2$, whereas $\zeta _\mathbb{K} (s)$ has a simple pole at $s=1$. We need to also consider the contribution of the non-trivial
 zeros of $\zeta _\mathbb{K}(s)$ and of the zero of $\zeta _\mathbb{K}(s)$ at $s=0$ of order $r=r_1+r_2-1$. Thus, in our contour the integrand of {\eqref{xr2}} has a poles of order $r$ at $s=1$ and $s=0$  as well as poles at the non-trivial zeros of $\zeta _\mathbb{K} (s)$.
 
 Applying Cauchy's residue theorem, we get 
 \begin{align}
 	\label{xr3}
 	&\frac{1}{{2\pi i}}\left [ {\int \limits _{ c - iT}^{d - iT} {+\int
 			\limits _{d - iT}^{d + iT} {+\int \limits _{d + iT}^{ c + iT} {+\int
 					\limits _{ c + iT}^{ c - iT} {} } } } } \right ]{\left ( {
 			\frac{{|{d_\mathbb{K}}|}}{{{4^{{r_{2}}}}{\pi ^{n}}}}} \right )^{
 			\frac{1}{2} - s}}{\alpha ^{ - s}}
 	\frac{{{\Gamma ^{{r_{1}}}}\left ( {\frac{{1 - s}}{2}} \right ){\Gamma ^{{r_{2}}}}\left ( {1 - s} \right )}}{{{\zeta _\mathbb{K}}(s)}}ds
 	\nonumber
 	\\
 	& = {\left . {\frac{1}{{(r - 1)!}}
 			\frac{{{d^{r - 1}}}}{{d{s^{r - 1}}}}{{\left ( {s - 1} \right )}^{r}}{{
 					\left ( {\frac{{|{d_\mathbb{K}}|}}{{{4^{{r_{2}}}}{\pi ^{n}}}}}
 					\right )}^{\frac{1}{2} - s}}{\alpha ^{ - s}}
 			\frac{{{\Gamma ^{{r_{1}}}}\left ( {\frac{{1 - s}}{2}} \right ){\Gamma ^{{r_{2}}}}\left ( {1 - s} \right )}}{{{\zeta _\mathbb{K}}(s)}}}
 		\right |_{s = 1}}
 	\nonumber
 	\\
 	& \quad+ {\left. {\frac{1}{{(r - 1)!}}\frac{{{d^{r - 1}}}}{{d{s^{r - 1}}}} {{ s  }^r}{{\left( {\frac{{|{d_\mathbb{K}}|}} {{{4^{{r_2}}}{\pi ^n}}}} \right)}^{\frac{1}{2} - s}}{\alpha ^{ - s}}\frac{{{\Gamma ^{{r_1}}}\left( {\frac{{1 - s}}{2}} \right){\Gamma ^{{r_2}}}\left( {1 - s} \right)}}{{{\zeta _\mathbb{K}}(s)}}} \right|_{s = 0}} \nonumber
 	\\
 	& \quad+ \sum \limits _\rho {{{\left ( {
 					\frac{{|{d_\mathbb{K}}|}}{{{4^{{r_{2}}}}{\pi ^{n}}}}} \right )}^{
 				\frac{1}{2} - \rho }}{\alpha ^{ - \rho }}
 		\frac{{{\Gamma ^{{r_{1}}}}\left ( {\frac{{1 - \rho }}{2}} \right ){\Gamma ^{{r_{2}}}}\left ( {1 - \rho } \right )}}{{{\zeta _\mathbb{K}}'(\rho )}}}
 	.
 \end{align}
Consider the integrals along the horizontal segments of the rectangular contour. By Stirling's formula and Lemma \ref{lemma}, as $|T|\rightarrow \infty$,
 $${\left( {\frac{{|{d_\mathbb{K}}|}}{{{4^{{r_2}}}{\pi ^n}}}} \right)^{\frac{1}{2} - s}}{\alpha ^{ - s}}\frac{{{\Gamma ^{{r_1}}}\left( {\frac{{1 - s}}{2}} \right){\Gamma ^{{r_2}}}\left( {1 - s} \right)}}{{{\zeta _\mathbb{K}}(s)}} = O\left( {e^{\left( {{A_2} - \frac{\pi }{4}\left( {{r_1} + 2{r_2}} \right)} \right)|T|}}\right) .$$ 
Since $A_2 < {\pi \over 4}$ and $r_1+2r_2 =n\geq 1$, as $|T|\rightarrow \infty$, we see that the integrals along the horizontal segments in \eqref{xr3} go to zero.
Thus,
{\allowdisplaybreaks\begin{align}
	\label{xr4}
	\frac{1}{{2\pi i}}\bigg[ \int \limits _{(d)} -\int \limits _{ (c)}
	\bigg]& {{\left ( {
				\frac{{|{d_\mathbb{K}}|}}{{{4^{{r_{2}}}}{\pi ^{n}}}}} \right )}^{
			\frac{1}{2} - s}}{\alpha ^{ - s}}
	\frac{{{\Gamma ^{{r_{1}}}}\left ( {\frac{{1 - s}}{2}} \right ){\Gamma ^{{r_{2}}}}\left ( {1 - s} \right )}}{{{\zeta _\mathbb{K}}(s)}}ds
	\nonumber
	\\
	= & \frac{1}{{(r - 1)!}}\left .\frac{{{d^{r - 1}}}}{{d{s^{r - 1}}}}{{
			\left ( {s - 1} \right )}^{r}}{{\left ( {
				\frac{{|{d_\mathbb{K}}|}}{{{4^{{r_{2}}}}{\pi ^{n}}}}} \right )}^{
			\frac{1}{2} - s}}{\alpha ^{ - s}}
	\frac{{{\Gamma ^{{r_{1}}}}\left ( {\frac{{1 - s}}{2}} \right ){\Gamma ^{{r_{2}}}}\left ( {1 - s} \right )}}{{{\zeta _\mathbb{K}}(s)}}
	\right |_{s = 1}
	\nonumber
	\\
	& +{\left. {\frac{1}{{(r - 1)!}}\frac{{{d^{r - 1}}}}{{d{s^{r - 1}}}} {{ s  }^r}{{\left( {\frac{{|{d_\mathbb{K}}|}} {{{4^{{r_2}}}{\pi ^n}}}} \right)}^{\frac{1}{2} - s}}{\alpha ^{ - s}}\frac{{{\Gamma ^{{r_1}}}\left( {\frac{{1 - s}}{2}} \right){\Gamma ^{{r_2}}}\left( {1 - s} \right)}}{{{\zeta _\mathbb{K}}(s)}}} \right|_{s = 0}} \nonumber
	\\
	&+ \sum \limits _\rho {{{\left ( {
					\frac{{|{d_\mathbb{K}}|}}{{{4^{{r_{2}}}}{\pi ^{n}}}}} \right )}^{
				\frac{1}{2} - \rho }}{\alpha ^{ - \rho }}
		\frac{{{\Gamma ^{{r_{1}}}}\left ( {\frac{{1 - \rho }}{2}} \right ){\Gamma ^{{r_{2}}}}\left ( {1 - \rho } \right )}}{{{\zeta _\mathbb{K}}^\prime (\rho )}}}
	.
\end{align}}
Now consider the first integral on the left-hand side of \eqref{xr4}. Suppose $\eta={4^{r_2}\pi^n/ {|d_\mathbb{K}|}}$ and $s=1-w$. Then, 
{\allowdisplaybreaks\begin{align*}
	\frac{1}{{2\pi i}} \int\limits_{ (d)} {{{\left( {\frac{{|{d_\mathbb{K}}|}}{{{4^{{r_2}}}{\pi ^n}}}} \right)}^{\frac{1}{2} - s}}{\alpha ^{ - s}}\frac{{{\Gamma ^{{r_1}}}\left( {\frac{{1 - s}}{2}} \right){\Gamma ^{{r_2}}}\left( {1 - s} \right)}}{{{\zeta _\mathbb{K}}(s)}}ds}& = \frac{1}{{2\pi i}}\frac{1}{{\sqrt \eta }}\int\limits_{ (d)} {{{\left( {\frac{\eta}{\alpha }} \right)}^s}\frac{{{\Gamma ^{{r_1}}}\left( {\frac{{1 - s}}{2}} \right){\Gamma ^{{r_2}}}\left( {1 - s} \right)}}{{{\zeta _\mathbb{K}}(s)}}ds} \nonumber \\  
	&  =   \frac{1}{{2\pi i}}\frac{1}{{\sqrt \eta}}\int\limits_{(c')} {{{\left( {\frac{\eta}{\alpha }} \right)}^{1 - w}}\frac{{{\Gamma ^{{r_1}}}\left( {\frac{w}{2}} \right){\Gamma ^{{r_2}}}\left(w \right)}}{{{\zeta _\mathbb{K}}(1 - w)}}dw} ,
\end{align*}}
where $-{1 \over 2}<c'=1-d<0$. Therefore using the fact $\alpha \beta = \eta$, we deduce that
\begin{align}\label{xr5}
	  \frac{1}{{2\pi i}} \int\limits_{ (d)} {{{\left( {\frac{{|{d_\mathbb{K}}|}}{{{4^{{r_2}}}{\pi ^n}}}} \right)}^{\frac{1}{2} - s}}{\alpha ^{ - s}}\frac{{{\Gamma ^{{r_1}}}\left( {\frac{{1 - s}}{2}} \right){\Gamma ^{{r_2}}}\left( {1 - s} \right)}}{{{\zeta _\mathbb{K}}(s)}}ds}&  =  \frac{1}{{2\pi i}}\frac{\beta }{{\sqrt \eta }}\int\limits_{\left( {c'} \right)} {{{ \beta  }^{ - w}}{\Gamma ^{{r_1}}}\left( {\frac{w}{2}} \right){\Gamma ^{{r_2}}}\left(w \right)\left( {\sum\limits_{n = 1}^\infty  {\frac{{{b_n}}}{{{n^{1 - w}}}}} } \right)dw} \nonumber \\
	  &  =   \frac{\beta }{{\sqrt \eta }}\sum\limits_{n = 1}^\infty  {\frac{{{b_n}}}{n}\frac{1}{{2\pi i}}\int\limits_{\left( {c'} \right)} {{\Gamma ^{{r_1}}}\left( {\frac{w}{2}} \right){\Gamma ^{{r_2}}}\left( w \right){{\left( {\frac{\beta }{n}} \right)}^{ - w}}dw} }   \nonumber \\
	    &  =  \frac{\beta }{{\sqrt \eta }}\sum\limits_{n = 1}^\infty  {\frac{{{b_n}}}{n}{Z_{{r_1},{r_2}}}\left( {\frac{\beta }{n}} \right)}  .
\end{align}
From \eqref{xr2}, \eqref{xr4}, and \eqref{xr5}, we get 
\begin{align*}
	-\sum \limits _{n = 1}^\infty {\frac{{{b_{n}}}}{n}{Z_{{r_{1}},{r_{2}}}}
		\left ( {\frac{\alpha }{n}} \right )} & +
	\frac{\beta }{{\sqrt \eta }}\sum \limits _{n = 1}^\infty {
		\frac{{{b_{n}}}}{n}{Z_{{r_{1}},{r_{2}}}}\left ( {\frac{\beta }{n}}
		\right )}
	\nonumber
	\\
	&= {\left . {\frac{1}{{(r - 1)!}}\frac{{{d^{r - 1}}}}{{d{s^{r - 1}}}}{{
					\left ( {s - 1} \right )}^{r}}{{\eta }^{ s-\frac{1}{2}}}{\alpha ^{ - s}}
			\frac{{{\Gamma ^{{r_{1}}}}\left ( {\frac{{1 - s}}{2}} \right ){\Gamma ^{{r_{2}}}}\left ( {1 - s} \right )}}{{{\zeta _\mathbb{K}}(s)}}}
		\right |_{s = 1}}
	\nonumber
	\\
	& \quad + {\left . {\frac{1}{{(r - 1)!}}\frac{{{d^{r - 1}}}}{{d{s^{r - 1}}}}{{
					s}^{r}}{{\eta }^{ s-\frac{1}{2}}}{\alpha ^{ - s}}
			\frac{{{\Gamma ^{{r_{1}}}}\left ( {\frac{{1 - s}}{2}} \right ){\Gamma ^{{r_{2}}}}\left ( {1 - s} \right )}}{{{\zeta _\mathbb{K}}(s)}}}
		\right |_{s = 0}}
	\nonumber
	\\
	& \quad + \sum \limits _\rho {{{ \eta }^{\rho - \frac{1}{2}}}{\alpha ^{
				- \rho }}
		\frac{{{\Gamma ^{{r_{1}}}}\left ( {\frac{{1 - \rho }}{2}} \right ){\Gamma ^{{r_{2}}}}\left ( {1 - \rho } \right )}}{{{\zeta ^\prime _\mathbb{K}} (\rho )}}}
	.
\end{align*}
Multiplying both sides by $-\sqrt{\alpha}$ and putting $\eta=\alpha \beta$, we arrive at \eqref{rhldedeeqn}. 
This completes the proof of Theorem \ref{rhldede}.
          
\subsection{Corollaries of Theorem \ref{rhldede}}
Our first corollary retrieves the statement of the Ramanujan-Hardy-Littlewood conjecture as stated in \eqref{mrhl}.      
 \begin{corollary}
 	Assume that the series $\sum_{\rho}\left(\Gamma{\left(\frac{1-\rho}{2}\right)}/\zeta^{'}(\rho)\right)x^{\rho}$ converges, where $\rho$ runs through the non-trivial zeros of $\zeta(s)$ and $x$ denotes a positive real number. Suppose that the non-trivial zeros of $\zeta(s)$ are simple. Let $\alpha,\beta>0$ such that $\alpha\beta =\pi$. Then,
 	\begin{equation}\label{rhlrz}
 		\sqrt{\alpha}\sum\limits_{n = 1}^\infty  {\frac{{\mu \left( n \right)}}{n}} {e^{ - {{\left( {\frac{\alpha }{n}} \right)}^2}}}
 		- \sqrt{\beta} \sum\limits_{n = 1}^\infty  {\frac{{\mu \left( n \right)}}{n}} {e^{ - {{\left( {\frac{\beta }{n}} \right)}^2}}} =-\frac{1}{2\sqrt{\beta}}\sum_{ \rho }\frac{\Gamma \left( \frac{1-\rho}{2} \right)}{\zeta ' (\rho)}\beta ^ \rho .
 	\end{equation}
 \end{corollary}	
\emph{Proof.}
	Let $\mathbb{K}=\mathbb{Q}$. We have $d_\mathbb{K}=1$, $r_1=1$, $r_2=0$. Moreover $Z_{1,0}(x)=2(e^{-x^2}-1)$ and $\eta=\alpha \beta=\pi$. 
	Substituting these values in \eqref{rhldedeeqn}, we get \eqref{rhlrz}.
\begin{corollary}\label{cor2}
	Let $\mathbb{K}=\mathbb{Q}(\sqrt{d})$ be a real quadratic field, where d is a square-free positive integer. Let $h$ and $R$ be the class number and regulator of the field respectively. Assume the convergence of the series ${{\sum_{ \rho }x ^{\rho}{\Gamma^2\left(\frac{1-\rho}{2}\right)}/{\zeta'_\mathbb{K}(\rho)}}}$, where $x>0$ and $\rho$ runs through the non-trivial zeros of $\zeta_\mathbb{K}(s)$. Suppose that the non-trivial zeros of $\zeta_\mathbb{K}(s)$ are simple. Let $\alpha, \beta>0 $ such that $\alpha \beta = {\pi^2}/{|d_\mathbb{K}|}$. Then, we have 
{\allowdisplaybreaks\begin{align}
		\label{corbes}
		& \sqrt{\alpha}\sum \limits _{n = 1}^\infty {\frac{{{b_{n}}}}{n}{
				\left (K_{0}\left ( {\frac{2\alpha }{n}} \right )+\gamma +\log
				\frac{\alpha}{n}\right )}} - \sqrt{\beta}\sum \limits _{n = 1}^
		\infty {\frac{{{b_{n}}}}{n}{\left (K_{0}\left ( {\frac{2\beta }{n}}
				\right )+\gamma +\log \frac{\beta}{n}\right )}}
		\nonumber
		\\
		&\qquad = -\frac{{\sqrt{\beta d_\mathbb{K}}}}{2hR} -{\pi \over \sqrt{\beta} {\zeta' _\mathbb{K}} (0 )} - {
			\frac{1}{4\sqrt{\beta}}}\sum \limits _\rho {{\beta ^{ \rho }}
			\frac{{{\Gamma ^{{2}}}\left ( {\frac{{1 - \rho }}{2}} \right )}}{{{\zeta ' _\mathbb{K}} (\rho )}}}
\end{align}}
\end{corollary}
\begin{proof}
	Since $\mathbb{K}=\mathbb{Q}(\sqrt{d})$, $r_1=2$, $r_2=0$. Hence by \eqref{bessel},
	\begin{align*}
		Z_{2,0}(x)&=\frac{1}{2\pi i}\int_{\left(c\right)}\Gamma^2\left(\frac{s}{2}\right)x^{-s}ds\\&=\frac{4}{2\pi i}\int_{\left(c\right)}\Gamma^2\left(\frac{s}{2}\right)(2x)^{-s}2^{s-2}ds \\&=4\left(K_0(2x)+\gamma+\log x\right).
	\end{align*} 
Therefore from \eqref{rhldedeeqn}, we have 
\begin{align}
	\label{eq8}
	& 4\sqrt{\alpha}\sum \limits _{n = 1}^\infty {\frac{{{b_{n}}}}{n}{
			\left (K_{0}\left ( {\frac{2\alpha }{n}} \right )+\gamma +\log
			\frac{\alpha}{n}\right )}} -4 \sqrt{\beta}\sum \limits _{n = 1}^
	\infty {\frac{{{b_{n}}}}{n}{\left (K_{0}\left ( {\frac{2\beta }{n}}
			\right )+\gamma +\log \frac{\beta}{n}\right )}}
	\nonumber
	\\
	&\qquad = -{ {\frac{1}{{\sqrt{\beta}}}\lim _{s\rightarrow 1}{\beta ^{ s}}
			\frac{(s-1){{\Gamma ^{{2}}}\left ( {\frac{{1 - s}}{2}} \right )}}{{{\zeta _\mathbb{K}}(s)}}}
	} -{ {\frac{1}{{\sqrt{\beta}}}\lim _{s\rightarrow 0}
			\frac{s\beta ^{ s}{{\Gamma ^{{2}}}\left ( {\frac{{1 - s}}{2}} \right )}}{{{\zeta _\mathbb{K}}(s)}}}
	} - {\frac{1}{\sqrt{\beta}}}\sum \limits _\rho {{\beta ^{ \rho }}
		\frac{{{\Gamma ^{{2}}}\left ( {\frac{{1 - \rho }}{2}} \right )}}{{{\zeta ' _\mathbb{K}} (\rho )}}}.
\end{align}
For the first term on the right-hand side of \eqref{eq8},
\begin{align}\label{eq9}
	{ {\frac{1}{{\sqrt{\beta}}}\lim_{s\rightarrow 1}{\beta ^{ s}}\frac{(s-1){{\Gamma ^{{2}}}\left( {\frac{{1 - s}}{2}} \right)}}{{{\zeta _\mathbb{K}}(s)}}} } 
	&= {{4{\sqrt{\beta}}}\lim_{s\rightarrow 1}{}\frac{{{\Gamma ^{{2}}}\left( {\frac{{3 - s}}{2}} \right)}}{{(s-1){\zeta _\mathbb{K}}(s)}}}  \nonumber \\
	&= \frac{{2\sqrt{\beta d_\mathbb{K}}}}{hR}.
\end{align}
Hence, from \eqref{eq8} and \eqref{eq9}, we obtain \eqref{corbes}.
\end{proof}
\begin{corollary}
		Let $\mathbb{K}=\mathbb{Q}(\sqrt{-d})$ be an imaginary quadratic field, where d is a square-free positive integer. Let $h$ and $R$ be the class number and regulator of the field respectively. Assume the convergence of the series $\sum_{ \rho }x ^{\rho}{\Gamma{(1-\rho)}}/{\zeta'_{\mathbb{K}}(\rho)}$, where $x>0$ and $\rho$ runs through the non-trivial zeros of $\zeta_{\mathbb{K}}(s)$. Suppose that the non-trivial zeros of $\zeta_\mathbb{K}(s)$ are simple. Let $\alpha,  \beta>0 $, such that $\alpha \beta = \frac{4\pi^2}{|d_\mathbb{K}|}$, then 
			$$ \sqrt{\alpha}\sum\limits_{n = 1}^\infty  {\frac{{b_n}}{n}} {e^{ - {{ {\frac{\alpha }{n}} }}}}
		- \sqrt{\beta} \sum\limits_{n = 1}^\infty  {\frac{{b_n}}{n}} {e^{ - {{ {\frac{\beta }{n}} }}}} =-\frac{1}{2\sqrt{\beta}}\sum_{ \rho }\frac{\Gamma \left( {1-\rho} \right)}{\zeta '_{\mathbb{K}} (\rho)}\beta ^ \rho .$$
\end{corollary}
\begin{proof}
	The result follows easily upon letting $r_1=0, \ r_2=1$, in which case $Z_{0,1}(x)=e^{-x}-1.$ 
\end{proof}
\section{Riesz-type criterion for the Generalized Riemann Hypothesis}
We begin this section with a heuristic resulting from  Theorem \ref{rhldede}, which motivates us to obtain a criterion equivalent to the Generalized Riemann Hypothesis for $\zeta_{\mathbb{K}}(s)$. In order to establish this heuristic, we need the following bound on $Z_{r_1,r_2}(x)$.



\begin{lemma}
	For any $c$ such that $-\frac{1}{2}<c<0$ and any $x>0$,  
	\begin{align}\label{eq10}
		Z_{r_1,r_2}\left(x \right) =O_{r_1,r_2}\left(x ^{-c}\right).
	\end{align}
\end{lemma}
\begin{proof}
	We have, using the functional equation of $\Gamma(s)$,
	\begin{align*}
		Z_{r_1,r_2}\left(x \right)
		&= \frac{2^{r_1}}{2\pi i}\int_{\left(  c\right) }\frac{\Gamma^{r_1}\left( \frac{s}{2}+1\right) \Gamma^{r_2}(s+1)}{s^{r_1+r_2}}x^{-s}ds. \nonumber\\
	&= \frac{2^{r_1}}{2\pi }\int_{-\infty }^{\infty}\frac{\Gamma^{r_1}\left( \frac{c}{2}+1+i\frac{t}{2}\right) \Gamma^{r_2}\left( c+1+it\right) }{\left(-\frac{1}{2}+it \right) ^{r_1+r_2}}x^{-c-it}dt.
	\end{align*}
	Using Stirling's formula and the Cauchy-Schwarz inequality, we get
	\begin{align*}
		Z_{r_1,r_2}\left(x \right)&= O\left(\int_{-\infty }^{\infty}\frac{|t|^{\frac{c(r_1+2r_2)}{2}+{r_1+r_2 \over 2}}e^{-{\pi \over 4}|t|(r_1+2r_2)}}{(1+|t|)^{r_1+r_2}}x^{-c }dt \right) =O_{r_1,r_2}\left(x^{-c} \right),
	\end{align*}
since $c\in \left( -{1\over 2},0\right) $.
\end{proof}

 From \eqref{rhldedeeqn} and \eqref{defp}, for $\alpha \beta ={4^{r_2}\pi^n/ {|d_\mathbb{K}|}}$, we have
\begin{align}
	\label{hue}
	\sqrt{\alpha} \mathcal{P} _{r_{1},r_{2}}(\alpha ^{2})-\sqrt{\beta}
	\mathcal{P} _{r_{1},r_{2}}(\beta ^{2}) = &-{\left . {
			\frac{1}{{\sqrt{\beta}(r - 1)!}}\frac{{{d^{r - 1}}}}{{d{s^{r - 1}}}}{{
					\left ( {s - 1} \right )}^{r}}{\beta ^{ s}}
			\frac{{{\Gamma ^{{r_{1}}}}\left ( {\frac{{1 - s}}{2}} \right ){\Gamma ^{{r_{2}}}}\left ( {1 - s} \right )}}{{{\zeta _\mathbb{K}}(s)}}}
		\right |_{s = 1}}
	\nonumber
	\\
	& - {\left. {\frac{1}{{\sqrt{\beta}(r - 1)!}}\frac{{{d^{r - 1}}}}{{d{s^{r - 1}}}}{{ {s} }^r}{\beta ^{ s}}\frac{{{\Gamma ^{{r_1}}}\left( {\frac{{1 - s}}{2}} \right){\Gamma ^{{r_2}}}\left( {1 - s} \right)}}{{{\zeta _\mathbb{K}}(s)}}} \right|_{s = 0}} \nonumber
	\\
	& - \sum \limits _\rho {{\beta ^{ \rho -{\frac{1}{2}} }}
		\frac{{{\Gamma ^{{r_{1}}}}\left ( {\frac{{1 - \rho }}{2}} \right ){\Gamma ^{{r_{2}}}}\left ( {1 - \rho } \right )}}{{{\zeta ' _\mathbb{K}} (\rho )}}}.
\end{align}
Recalling that $\mathcal{P} _{r_1,r_2}(\alpha ^2)=O_{r_1,r_2}(\alpha ^{-c})$ from \eqref{eq10}, The first term on the left-hand side of \eqref{hue} goes to zero as $\alpha \rightarrow 0$. We now assume GRH and the convergence of $\sum\limits_\rho  {\frac{{{\Gamma ^{{r_1}}}\left( {\frac{{1 - \rho }}{2}} \right){\Gamma ^{{r_2}}}\left( {1 - \rho } \right)}}{{{\zeta' _\mathbb{K}} (\rho )}}}$. Now suppose $M(s)={\left( {s - 1} \right)}^r\frac{{{\Gamma ^{{r_1}}}\left( {\frac{{1 - s}}{2}} \right){\Gamma ^{{r_2}}}\left( {1 - s} \right)}}{{{\zeta _\mathbb{K}}(s)}}$ and $D_i=M^{(i)}(1)$, where $M^{(i)}$ denotes the $i^{th}$ derivative of $M$. Then by the product rule
\begin{align*}
	{\left. {\frac{1}{{\sqrt{\beta}(r - 1)!}}\frac{{{d^{r - 1}}}}{{d{s^{r - 1}}}}{{\left( {s - 1} \right)}^r}{\beta ^{ s}}\frac{{{\Gamma ^{{r_1}}}\left( {\frac{{1 - s}}{2}} \right){\Gamma ^{{r_2}}}\left( {1 - s} \right)}}{{{\zeta _\mathbb{K}}(s)}}} \right|_{s = 1}} &= {\left. {\frac{1}{{\sqrt{\beta}(r - 1)!}}\frac{{{d^{r - 1}}}}{{d{s^{r - 1}}}}{\beta ^{ s}}M(s)} \right|_{s = 1}}
	\\ 
	&=    \frac{1}{{\sqrt{\beta}(r - 1)!}}\sum_{i=0}^{r-1}\left(\begin{matrix}
		r-1\\i
	\end{matrix} \right) \beta \left( \log \beta \right)^{r-1-i} D_i 
	\\
	&= \frac{\sqrt{\beta}}{{(r - 1)!}}\sum_{i=0}^{r-1}\left(\begin{matrix}
		r-1\\i
	\end{matrix} \right)  \left( \log \beta \right)^{r-1-i} D_i .
\end{align*}
	Similarly, with $N(s)={s }^r\frac{{{\Gamma ^{{r_1}}}\left( {\frac{{1 - s}}{2}} \right){\Gamma ^{{r_2}}}\left( {1 - s} \right)}}{{{\zeta _\mathbb{K}}(s)}}$ and $E_i=N^{(i)}(0)$, where $N^{(i)}$ denotes the $i^{th}$ derivative of $N$, we have 
\begin{align*}
	{\left. {\frac{1}{{\sqrt{\beta}(r - 1)!}}\frac{{{d^{r - 1}}}}{{d{s^{r - 1}}}}{{ s }^r}{\beta ^{ s}}\frac{{{\Gamma ^{{r_1}}}\left( {\frac{{1 - s}}{2}} \right){\Gamma ^{{r_2}}}\left( {1 - s} \right)}}{{{\zeta _\mathbb{K}}(s)}}} \right|_{s = 0}} =  \frac{1}{{\sqrt{\beta} (r - 1)!}}\sum_{i=0}^{r-1}\left(\begin{matrix}
		r-1\\i
	\end{matrix} \right)  \left( \log \beta \right)^{r-1-i} E_i .
\end{align*}
Consequently, as $\alpha \rightarrow 0$, or equivalently, as $\beta \rightarrow \infty$, for $r\neq 0$, \eqref{hue} becomes 
\begin{align*}
	\mathcal{P} _{r_1,r_2}(\beta^2)=\frac{1}{{(r - 1)!}}\sum_{i=0}^{r-1}\left(\begin{matrix}
		r-1\\i
	\end{matrix} \right)  \left( \log {\beta} \right)^{r-1-i} D_i+O_{r_1,r_2}\left( \beta^{-\frac{1}{2}} \right),
\end{align*}
whereas for $r = 0$, we simply have $	\mathcal{P} _{r_1,r_2}(\beta^2)=O_{r_1,r_2}\left( \beta^{-\frac{1}{2}} \right).$

 This heuristic assumes the convergence of the series $\sum\limits_\rho  {\frac{{{\Gamma ^{{r_1}}}\left( {\frac{{1 - \rho }}{2}} \right){\Gamma ^{{r_2}}}\left( {1 - \rho } \right)}}{{{\zeta' _\mathbb{K}} (\rho )}}}$. Without this assumption, for $r\neq 0$, the main term in the estimate of $\mathcal{P} _{r_1,r_2}(\beta^2)$ is of the order of $\left( \log\beta\right) ^r $ and the error term is of the order of $\beta^{-{1\over 2}+\delta} $ for every $\delta>0 $, whereas, for $r= 0$, $	\mathcal{P} _{r_1,r_2}(\beta^2)=O_{r_1,r_2}\left( \beta^{-\frac{1}{2}+\delta} \right).$ This is shown in Theorem \ref{rtc}. In order to prove this theorem, we need the following lemma.

\begin{lemma}
	Let $0 <\textup{Re}(s)<{1 \over 2}$. Then for any non-negative integers $r_1, r_2 $,
	\begin{equation}\label{lemma2}
		\int_{0}^{\infty}y^{-s-1}\mathcal{P} _{r_1,r_2}(y)dy=2\frac{\Gamma^{r_1}(-s)\Gamma^{r_2}(-2s)}{\zeta_\mathbb{K}(2s+1)}.
	\end{equation}
\end{lemma}
\begin{proof}
	Let 
	\begin{align}\label{eq11}
		\varphi(s,r_1,r_2)=\int_{0}^{\infty}y^{-s-1}\mathcal{P} _{r_1,r_2}(y)dy.
	\end{align}
Put $y=\frac{x}{\mathcal{N}(\mathfrak{a})^2}$ in \eqref{eq11},  where $\mathcal{N}(\mathfrak{a})$ is the norm of an ideal $\mathfrak{a}$ of $\mathcal{O}_\mathbb{K}$. Then 
\begin{align*}
	\mathcal{N}(\mathfrak{a})^{-2s-1}\varphi(s,r_1,r_2)=\int_{0}^{\infty}\frac{x^{-s-1}}{\mathcal{N}(\mathfrak{a})}\mathcal{P} _{r_1,r_2}\left( \frac{x}{\mathcal{N}(\mathfrak{a})^2}\right) dx.
\end{align*}
Summing over all non-zero ideals of $\mathcal{O}_\mathbb{K}$, we have
\begin{align*}
	\zeta_\mathbb{K}(2s+1)\varphi(s,r_1,r_2)=\sum_{\mathfrak{a}}\int_{0}^{\infty}\frac{x^{-s-1}}{\mathcal{N}(\mathfrak{a})}\mathcal{P}_{r_1,r_2}\left(\frac{x}{\mathcal{N}(\mathfrak{a})^2}\right) dx.
\end{align*}
Using the Weierstrass M-test and the Lebesgue dominated convergence theorem, we get
\begin{align*}
		\zeta_\mathbb{K}(2s+1)\varphi(s,r_1,r_2)&=\int_{0}^{\infty}x^{-s-1}\sum_{\mathfrak{a} }\frac{1}{\mathcal{N}(\mathfrak{a})}\mathcal{P}_{r_1,r_2}\left(\frac{x}{\mathcal{N}(\mathfrak{a})^2}\right) dx\\
		&=\int_{0}^{\infty}x^{-s-1}\sum_{\mathfrak{a} }\frac{1}{\mathcal{N}(\mathfrak{a})}\frac{1}{{2\pi i}}\int\limits_{\left( {c} \right)} \frac{{{\Gamma ^{{r_1}}}\left( {\frac{s'}{2}} \right){\Gamma ^{{r_2}}}\left( s' \right)}}{{{\zeta _\mathbb{K}}(1 - s')}}\frac{x^{-\frac{s'}{2}}}{{\mathcal{N}(\mathfrak{a})^{-s'}}}ds'  dx\\
		&=\frac{1}{{2\pi i}}\int_{0}^{\infty}x^{-s-1} \int\limits_{\left( {c} \right)} \frac{{{\Gamma ^{{r_1}}}\left( {\frac{s'}{2}} \right){\Gamma ^{{r_2}}}\left( s' \right)}x^{-\frac{s'}{2}}{{\zeta _\mathbb{K}}(1 - s')}}{{{\zeta _\mathbb{K}}(1 - s')}}ds' dx.
\end{align*}
Substitute $s'$ by $-2z$ to get
\begin{align*}
		\zeta_\mathbb{K}(2s+1)\varphi(s,r_1,r_2)&=\frac{2}{{2\pi i}}\int_{0}^{\infty}x^{-s-1} \int\limits_{\left(- {c \over 2} \right)} {{\Gamma ^{{r_1}}}\left( {-z} \right){\Gamma ^{{r_2}}}\left(- 2z \right)}x^{z}dz dx.
\end{align*}
Replacing $x$ by $\frac{1}{x}$ and using the Mellin inversion theorem \cite[p. 341-343]{Mcla}, we have
\begin{align*}
	\zeta_\mathbb{K}(2s+1)\varphi(s,r_1,r_2)
	&=2\int_{0}^{\infty}{x^{s-1}}\bigg( \frac{1}{{2\pi i}}\int\limits_{\left(- {c \over 2} \right)} {{\Gamma ^{{r_1}}}\left( {-z} \right){\Gamma ^{{r_2}}}\left(- 2z \right)}x^{-z}dz\bigg) dx\\
	&=2{{\Gamma ^{{r_1}}}\left( {-s} \right){\Gamma ^{{r_2}}}\left(- 2s \right)}.
\end{align*}
This completes the proof.\end{proof}
\begin{remark*}
	If we let $r_1=1, \ r_2=0$ in the above lemma, we recover a result of Hardy and Littlewood \cite[Equation (2.544)]{hl}: 
	\begin{equation*}
\int_{0}^{\infty}y^{-s-1}\sum_{k=1}^{\infty}\frac{\mu(k)}{k}e^{-{y/k^2}}dy=\frac{\Gamma(-s)}{\zeta(2s+1)}.
	\end{equation*}
\end{remark*}

\subsection{Proof of Theorem \ref{rtc}}
	We first prove part (1). Multiplying both sides by $s^{r_1+r_2}$  in \eqref{lemma2}, we get
\begin{align}\label{Analy}
		s^{r+1}\zeta_\mathbb{K}(2s+1)\int_{0}^{\infty}y^{-s-1}\mathcal{P} _{r_1,r_2}(y)dy
		=\frac{(-1)^{r_1+r_2}}{2^{r_2-1}}{{\Gamma}^{{r_1}} (1-s){\Gamma}^{r_2} (1-2s)}.
	\end{align}

We now show that \eqref{Analy} also holds in $-{1 \over 4}<\textup{Re}(s)\leq 0$, provided $\mathcal{P} _{r_1,r_2}(y)=O_{r_1,r_2}\left( y^{-\frac{1}{4}+\delta} \right)$ as $y \rightarrow \infty$ for all $\delta >0$. Since $\zeta_\mathbb{K}(2s+1)$ has a simple pole at $s=0$, it is clear that $s^{r+1}\zeta_\mathbb{K}(2s+1)$ is entire. Next,
split the integral $\int_{0}^{\infty}y^{-s-1}\mathcal{P} _{r_1,r_2}(y)dy$ into two parts, one from 0 to 1 and another from 1 to $\infty$. 
For the second integral, the bound $\mathcal{P} _{r_1,r_2}(y)\ll_{r_1,r_2} y^{-\frac{1}{4}+\delta} $ implies that the integral is analytic in the region  $-\frac{1}{4}<$Re$(s)\leq0$. That the first integral is analytic in the same region is seen using \eqref{eq10}, since $-1/2 <c<0$. Also $\Gamma(1-s)$ and $\Gamma(1-2s)$ are analytic in $-\frac{1}{4}<$Re$(s)\leq0$.

By the principle of analytic continuation, \eqref{Analy} holds in the required region. Note that the right-hand side has no zeros in the region $-\frac{1}{4}<$Re$(s)<0$. Moreover, the integral on the left-hand side is analytic in the same region. This implies that $\zeta_\mathbb{K}(2s+1)$ does not vanish in this region.  This implies the Generalized Riemann Hypothesis and completes the proof of part (1).

{We will now prove part 2(a).} 
Let $M_{\mathbb{K}}(x)=\sum_{ n\leq x}b_n$. 
 We  first observe that  GRH implies  the bound \begin{equation}\label{estm}
	M_{\mathbb{K}}(x)\ll x^{\frac{1}{2}+\epsilon}.
\end{equation}
 for any $\epsilon >0 $, as $x\rightarrow \i$. This follows from Proposition 5.14 of \cite{Iwaniec-Kowalski} after checking  that $\zeta_\mathbb K(s)$  satisfies the required hypotheses stated on p. $94$ of \cite{Iwaniec-Kowalski}. Indeed, the conditions (5.1), (5.2) and the functional equation (equations (5.3) to (5.5)) in \cite[p.~94]{Iwaniec-Kowalski} essentially reiterate that  $\zeta_{\mathbb{K}}(s)$ lies in the Selberg class of functions with polynomial Euler product. This is well-known, see for instance, p.5 of \cite{Simonic}. The condition (3) in \cite[p.~94]{Iwaniec-Kowalski} also holds since $\zeta_{\mathbb{K}}(s)$ has conductor  $|d_\mathbb{K}|$,  as mentioned on p. 125 of \cite{Iwaniec-Kowalski}. The interested reader may also refer to \cite[p.~3]{minamide}. 

We will now use \eqref{estm} in our analysis below. Define $$M_{\mathbb{K}}(\nu,n)=\sum_{m=\nu }^{n}\frac{b_m}{m}.$$
By the partial summation formula, one can easily derive $M_{\mathbb{K}}(\nu,n)\ll_\epsilon  \nu^{-{1\over 2}+\epsilon} .$
Let $\nu=[\beta ^{1-\epsilon}]$. Then 
\begin{align*}
	 \mathcal{P} _{r_1,r_2}(\beta ^2)&=\sum\limits_{n = 1}^\infty  {\frac{{{b_n}}}{n}{Z_{{r_1},{r_2}}}\left( {\frac{\beta }{n}}  \right)}\nonumber  \\ 
	 &=\left[\sum_{n=1}^{\nu-1}+\sum_{n=\nu}^{\infty} \right] \frac{{{b_n}}}{n}{Z_{{r_1},{r_2}}}\left( {\frac{\beta }{n}} \right)=:P_1+P_2.
\end{align*}
We first handle $P_2$ as follows. We have,
\begin{align*}\label{p2}
	\sum_{n=\nu}^{N}\frac{{{b_n}}}{n}{Z_{{r_1},{r_2}}}\left( {\frac{\beta }{n}} \right)&=\frac{{{b_\nu}}}{\nu}{Z_{{r_1},{r_2}}}\left( {\frac{\beta }{\nu}} \right)+\sum_{ \nu<n\leq N }\left( M_{\mathbb{K}}(\nu,n)-M_{\mathbb{K}}(\nu,n-1)\right) {Z_{{r_1},{r_2}}}\left( {\frac{\beta }{n}}\right) \nonumber \\
	&=\sum_{ \nu-1<n\leq N-1 } M_{\mathbb{K}}(\nu,n) \left( {Z_{{r_1},{r_2}}}\left( {\frac{\beta }{n}}\right)-{Z_{{r_1},{r_2}}}\left( {\frac{\beta }{n+1}}\right)\right) + M_{\mathbb{K}}{(\nu,N)}{Z_{{r_1},{r_2}}}\left( {\frac{\beta }{N}} \right)  \nonumber \\ 
	&=\sum_{ \nu-1<n\leq N-1 } M_{\mathbb{K}}(\nu,n)  {Z'_{{r_1},{r_2}}}\left( {\frac{\beta }{\lambda_n}}\right) + O_\epsilon \left( \nu^{-{1\over 2}+\epsilon} \right)O_{r_1,r_2}\left(\frac{\beta}{N} \right)^{-c}
\end{align*}
In the final equality above, we have used \eqref{eq10} and the mean value theorem with $n<\lambda _n <n+1$. Letting $N\rightarrow \infty$, we get
 \begin{align*}
	\sum_{n=\nu}^{\infty}\frac{{{b_n}}}{n}{Z_{{r_1},{r_2}}}\left( {\frac{\beta }{n}} \right)
	&=\sum_{n=\nu}^{\infty} M_{\mathbb{K}}(\nu,n)  {Z'_{{r_1},{r_2}}}\left( {\frac{\beta }{\lambda_n}}\right) \nonumber \\
	&\ll_\epsilon    \nu^{-{1\over 2}+\epsilon} \sum_{n=\nu}^{\infty} {Z'_{{r_1},{r_2}}}\left( {\frac{\beta }{\lambda_n}}\right)    =   \nu^{-{1\over 2}+\epsilon}P_3.
\end{align*}
Now 
\begin{align*}
	{\frac{d }{dx}Z_{r_1,r_2}\left(\frac{\beta}{x} \right)}&= \frac{1}{2\pi i}\int_{\left( c\right) }\Gamma^{r_1}\left( \frac{s}{2}\right) \Gamma^{r_2}(s)\left(\frac{sx^{s-1}}{\beta^s} \right)ds\nonumber\\
	&= O\left(\int_{-\infty }^{\infty}\frac{|t|^{\frac{c(r_1+2r_2)}{2}+{r_1+r_2 \over 2}}e^{-{\pi \over 4}|t|(r_1+2r_2)}}{(1+|t|)^{r_1+r_2-1}}\left(\frac{\beta^{-c}}{x^{-c+1}} \right)dt \right) \nonumber\\
	&=O_{r_1,r_2}\left(\frac{\beta^{-c}}{x^{-c+1}} \right).
\end{align*}
Hence,
\begin{align*}
	P_3=\sum_{n=\nu}^{\infty} {Z'_{{r_1},{r_2}}}\left( {\frac{\beta }{\lambda_n}}\right)=O_{r_1,r_2}\left( \sum_{n=\nu}^{\infty}\left(\frac{\beta^{-c}}{n^{-c+1}} \right) \right) =O_{r_1,r_2}(\beta ^{-\epsilon c}),
\end{align*}
where we use the bound ${\beta \over \nu}=O(\beta ^\epsilon )$.
From \eqref{defz}, \eqref{defzz} and the residue theorem, we find that
\begin{align*}
	P_1&=\sum_{n=1}^{\nu-1}\frac{{{b_n}}}{n}{Z_{{r_1},{r_2}}}\left( {\frac{\beta }{n}} \right) \nonumber \\
	&=\sum_{n=1}^{\nu-1}\frac{{{b_n}}}{n}\left( {\tilde{Z}_{{r_1},{r_2}}}\left( {\frac{\beta }{n}} \right)-\text{Res}_{s=0}\Gamma^{r_1}\left( {s\over 2}\right) \Gamma^{r_2}(s) \left({{\beta}\over n} \right)^{-s} \right)
	\end{align*}
Using \eqref{estz}, we see that
\begin{align*}
	P_1&=O\left(  \sum_{n=1}^{\nu-1}\frac{{{b_n}}}{n}\exp\left( {-\sigma\left(\frac{\beta}{n4^{r_2}} \right)^{2\over \sigma}}\right) \left(\frac{\beta}{n4^{r_2}} \right)^{-{r\over \sigma}} \right)- \sum_{n=1}^{\nu-1}\frac{{{b_n}}}{n}\text{Res}_{s=0}\Gamma^{r_1}\left( {s\over 2}\right) \Gamma^{r_2}(s) \left({{\beta}\over n} \right)^{-s} \nonumber \\
	 &=O\left( \exp\left({-\sigma\left(\frac{\beta^\epsilon}{4^{r_{2}}} \right)^{2\over \sigma}}\right)  \sum_{n=1}^{\nu-1}\frac{{{b_n}}}{n}\left(\frac{\beta}{n4^{r_2}} \right)^{-{r\over \sigma}} \right)-\frac{2^{r_1}}{r!} \sum_{n=1}^{\nu-1}\frac{{{b_n}}}{n}\sum_{i=0}^{r}C_i\left( \begin{matrix}
		r \\ i
	\end{matrix}\right) \left(\log (n) - \log (\beta) \right)^{r-i}, 
\end{align*}
where $\sigma=[\mathbb{K}:\mathbb{Q}]$, $C_i=X^{(i)}_{r_1,r_2}(0)$, and $X^{(i)}_{r_1,r_2}(s)$ denotes the $i^{th}$ derivative of  $\Gamma^{r_1}\left( {s\over 2}+1\right) \Gamma^{r_2}(s+1)$. The first term goes to zero as $\beta \rightarrow \infty$. 
Combining the estimates for $P_1$ and $P_2$ and replacing $\beta ^2$ by $y$, we arrive at \eqref{estp}.

 We now prove part 2(b). Let $m_\mathbb{K}(x)=\sum_{n\leq x}{b_n \over n}.$ Using the bound in \eqref{estm} which follows from the GRH, 
 partial summation gives 
 \begin{align*}
 	m_\mathbb{K}(x)&=\int_{1}^{\infty}\frac{M_\mathbb{K}(t)}{t^2}dt-\int_{x}^{\infty}\frac{M_\mathbb{K}(t)}{t^2}dt+\frac{M_\mathbb{K}(x)}{x}\\&=C+O_\epsilon \left(x^{-{1\over 2}+\epsilon}\right),
 \end{align*}
where $C$ is the value of the convergent integral $\int_{1}^{\infty}\frac{M_\mathbb{K}(x)}{t^2}dt$. Since $\sum_{n\leq x}{b_n \over n}$ converges to zero as $x \rightarrow \infty$, we must have $C=0$, so that  $	m_\mathbb{K}(x)\ll_\epsilon x^{-{1\over 2}+\epsilon}$ under GRH.
Turning to \eqref{estp}, we see that for $r=0$, the main term is 
\begin{align*}
	-2^{r_2}\sum_{n=1}^{\lfloor y^{{1\over 2}-\epsilon}\rfloor}{b_n\over n}&\ll_\delta y^{-{1\over 4}+\delta},
\end{align*}
for any $\delta>0$. This completes the proof.




\begin{center}
	\textbf{Acknowledgements}
\end{center}

The authors sincerely thank the referee for insightful comments which improved the quality of the paper. They also thank Shigeru Kanemitsu and Makoto Minamide for interesting discussions. The first and third authors sincerely thank the MHRD SPARC project SPARC/2018-2019/P567/SL for the financial support. The first author's research was partially supported by the SERB-DST CRG grant CRG/2020/002367. The second author acknowledges the support of CSIR SPM Fellowship under the grant number SPM-06/1031(0281)/2018-EMR-I. The third author was partially supported by the SERB-DST grant ECR/2018/001566
and the DST INSPIRE Faculty Award Program  DST/INSPIRE/Faculty/Batch-13/2018.

\end{document}